\documentclass[12pt]{amsart}
\usepackage{hyperref}
\usepackage{amssymb}

\newtheorem{theorem}{Theorem}[section]
\newtheorem{definition}[theorem]{Definition}
\newtheorem{proposition}[theorem]{Proposition}
\newtheorem{lemma}[theorem]{Lemma}

\newtheorem{corollary}[theorem]{Corollary}
\newtheorem{conjecture}[theorem]{Conjecture}

\newtheorem{problem}[theorem]{Problem}

\begin{document}

\title{First order deformations of the Fourier matrix}

\author{Teodor Banica}
\address{T.B.: Department of Mathematics, Cergy-Pontoise University, 95000 Cergy-Pontoise, France. {\tt teo.banica@gmail.com}}


\begin{abstract}
The $N\times N$ complex Hadamard matrices form a real algebraic manifold $C_N$. The singularity at a point $H\in C_N$ is described by a filtration of cones $T^\times_HC_N\subset T^\circ_HC_N\subset T_HC_N\subset\widetilde{T}_HC_N$, coming from the trivial, affine, smooth and first order deformations. We study here these cones in the case where $H=F_N$ is the Fourier matrix, $(w^{ij})$ with $w=e^{2\pi i/N}$, our main result being a simple description of $\widetilde{T}_HC_N$. As a consequence, the rationality conjecture $dim_\mathbb R(\widetilde{T}_HC_N)=dim_\mathbb Q(\widetilde{T}_HC_N\cap M_N(\mathbb Q))$ holds at $H=F_N$.
\end{abstract}

\maketitle

\tableofcontents

\section*{Introduction}

A complex Hadamard matrix is a square matrix $H\in M_N(\mathbb C)$ whose entries are on the unit circle, $|H_{ij}|=1$, and whose rows are pairwise orthogonal. The basic example of such a matrix is the Fourier one, $F_N=(w^{ij})$ with $w=e^{2\pi i/N}$:
$$F_N=\begin{pmatrix}
1&1&1&\ldots&1\\
1&w&w^2&\ldots&w^{N-1}\\
\ldots&\ldots&\ldots&\ldots&\ldots\\
1&w^{N-1}&w^{2(N-1)}&\ldots&w^{(N-1)^2}
\end{pmatrix}$$

In general, the theory of complex Hadamard matrices can be regarded as a ``non-standard'' branch of discrete Fourier analysis. For a number of potential applications, to quantum physics and quantum information theory questions, see \cite{bb+}, \cite{tz1}, \cite{wer}.

The complex Hadamard matrices are also known to parametrize the pairs of orthogonal MASA in the simplest von Neumann algebra, $M_N(\mathbb C)$. This discovery of Popa \cite{pop} has led to deep connections with a number of related areas. See \cite{bbs}, \cite{jon}, \cite{jsu}, \cite{pop}.

One key problem in the area, raised by Jones in \cite{jon}, concerns the computation of the numbers $c_k=\dim P_k$, where $P=(P_k)$ is the planar algebra associated to $H$. These numbers, called quantum invariants of $H$, are in general extremely hard to compute. Very little is known about them, and in particular we have the following question:

\medskip

\noindent {\bf Problem.} {\em What is the relation between the quantum invariants of $H$, and the geometry of the complex Hadamard matrix manifold around $H$?}

\medskip

In order to further comment on this problem, let us first observe that the complex Hadamard matrix manifold $C_N$ can be defined as follows:
$$C_N=M_N(\mathbb T)\cap\sqrt{N}U_N$$

Thus $C_N$ is a real algebraic manifold, not smooth in general. The singularity at a given point $H\in C_N$ is best described by a filtration of cones, as follows: 
$$T^\times_HC_N\subset T^\circ_HC_N\subset T_HC_N\subset\widetilde{T}_HC_N$$

Here $T_HC_N$ is the tangent cone, and $T_H^\circ C_N$ is the affine tangent cone, obtained by restricting attention to the affine deformations. These cones live in the enveloping tangent space $\widetilde{T}_HC_N$, obtained by interesecting the tangent spaces to the smooth manifolds $M_N(\mathbb T)$ and $\sqrt{N}U_N$, and contain the trivial tangent cone $T_H^\times C_N$, consisting of vectors which are tangent to the trivial deformations, obtained by multiplying rows and columns.

In general, the computation of these cones is a quite difficult linear algebra problem. In the Fourier matrix case, however, we have the following key formula:
$$\dim(\widetilde{T}_HC_N)=\#\{(i,j)|H_{ij}=1\}$$

This result was established by Tadej and \.Zyczkowski in \cite{tz2}, with parts of it going back to Karabegov's paper \cite{kar}, and with the general case discussed in detail in \cite{ban}.

In this paper, following a number of supplementary ideas from \cite{bbe}, \cite{dit}, \cite{tz1}, we will obtain a finer result about $H=F_N$, directly in terms of $\widetilde{T}_HC_N$:

\medskip

\noindent {\bf Theorem.} {\em For $H=F_N$ the vectors $A\in\widetilde{T}_HC_N$ appear as plain sums of the following type, where the $L$ variables form dephased matrices $L^{GH}\in M_{G\times H}(\mathbb R)$:
$$A_{ij}=\sum_{G\times H\subset\mathbb Z_N}L^{GH}_{\varphi_G(i)\varphi_H(j)}$$
In particular, the rationality conjecture $dim_\mathbb R(\widetilde{T}_HC_N)=dim_\mathbb Q(\widetilde{T}_HC_N\cap M_N(\mathbb Q))$ holds.}

\medskip 

We refer to section 3 below for the precise formulation of this result, which requires a number of preliminaries, not to be explained in detail right now.

Let us go back now to the quantum algebra/algebraic geometry problem formulated above. This problem was recently investigated in \cite{ban}, with a proposal there involving Diaconis-Shahshahani type variables \cite{dsh}. The results obtained here suggest:

\medskip

\noindent {\bf Question.} {\em What algebraic and geometric invariants of $H$ are encoded by the statistics of the number of $1$ entries, over the equivalence class of $H$?}

\medskip

More precisely, consider for instance a usual Hadamard matrix, $H\in M_N(\pm 1)$. This matrix is of course described by the set of indices $E\subset\{1,\ldots,N\}\times\{1,\ldots,N\}$ telling us where the 1 entries are. What we propose here is a more geometric approach to $H$, by considering the function $\varphi:\mathbb Z_2^N\times\mathbb Z_2^N\to\mathbb N$ which counts the number of 1 entries of the various conjugates of $H$, obtained by switching signs on rows and columns:
$$\varphi(a,b)=\#\{(i,j)|a_ib_jH_{ij}=1\}$$

This construction can be generalized to the Butson matrices, and our claim is that $\varphi$, or just its probabilistic distribution $\mu$, should encode important information about $H$.

In general, the computation of $\mu$ is a difficult problem. In the real case, this is related to the Gale-Berlekamp game \cite{fsl}, \cite{rvi} and to various questions regarding the 0-1 matrices, cf. \cite{cra}, \cite{nra}, \cite{tvu}. We intend to come back to these questions in some future work.

The paper is organized as follows: 1-2 are preliminary sections, containing known material along with a number of new results, in 3 we state and prove our main results, and 4 contains the probabilistic speculations, and a few concluding remarks.

\medskip

\noindent {\bf Acknowledgements.} We would like to thank Ingemar Bengtsson, Alexander Karabegov, Ion Nechita, Joseph O'Rourke, Gerhard Paseman, Patrick Popescu-Pampu, Jean-Marc Schlenker, Ferenc Sz\"oll\H{o}si and Karol \.Zyczkowski for several useful discussions.

\section{Complex Hadamard matrices}

We consider in this paper various square matrices over the complex numbers, $H\in M_N(\mathbb C)$. The indices of our matrices will usually range in the set $\{0,1,\ldots,N-1\}$.

\begin{definition}
A complex Hadamard matrix is a matrix $H\in M_N(\mathbb C)$ whose entries are on the unit circle, $|H_{ij}|=1$, and whose rows are pairwise orthogonal.
\end{definition}

The basic example is the Fourier matrix $F_N=(w^{ij})$ with $w=e^{2\pi i/N}$. 

One way of constructing new examples is by taking tensor products, $(H\otimes K)_{ia,jb}=H_{ij}K_{ab}$. In matrix form, by using the lexicographic order on the double indices:
$$H\otimes K=\begin{pmatrix}H_{11}K&\ldots&H_{1N}K\\ \ldots&\ldots&\ldots\\ H_{N1}K&\ldots&H_{NN}K\end{pmatrix}$$

Observe that the Fourier matrix $F_N$ is nothing but the matrix of the discrete Fourier transform, over the cyclic group $\mathbb Z_N$. In fact, we have the following result:

\begin{proposition}
The Fourier matrix of a finite abelian group $G=\mathbb Z_{N_1}\times\ldots\times\mathbb Z_{N_k}$ is the complex Hadamard matrix $F_G=F_{N_1}\otimes\ldots\otimes F_{N_k}$. 
\end{proposition}

\begin{proof}
For a product of groups $G=G'\times G''$ we have $F_G=F_{G'}\otimes F_{G''}$, and together with the above observation regarding $F_N$, this gives the equality in the statement.
\end{proof}

As an example, for the Klein group $\mathbb Z_2\times\mathbb Z_2$ we obtain the following matrix:
$$F_2\otimes F_2=\begin{pmatrix}
1&1\\
1&-1
\end{pmatrix}
\otimes
\begin{pmatrix}
1&1\\
1&-1
\end{pmatrix}
=\begin{pmatrix}
1&1&1&1\\
1&-1&1&-1\\
1&1&-1&-1\\ 
1&-1&-1&1
\end{pmatrix}$$

The following notions will play a key role in what follows:

\begin{definition}
Let $H,K\in M_N(\mathbb C)$ be two complex Hadamard matrices.
\begin{enumerate}
\item $H$ is called dephased  if its first row and column consist of $1$ entries only. 

\item $H,K$ are called equivalent if one can pass from one to the other by permuting rows and columns, or by multiplying the rows and columns by numbers in $\mathbb T$.
\end{enumerate}
\end{definition}

In other words, we use the equivalence relation on the $N\times N$ matrices coming from the action $T_{(A,B)}(H)=AHB^*$ of the group $G=(K_N\times K_N)/\mathbb T$, where $K_N=\mathbb T\wr S_N$ is the group of permutation matrices with entries multiplied by elements of $\mathbb T$. Observe that any complex Hadamard matrix can be assumed, up to equivalence, to be dephased.

Given a matrix $Q\in M_2(\mathbb T)$, written $Q=(^a_c{\ }^b_d)$, we can form the following matrix:
$$F_2\!{\ }_Q\!\otimes F_2=\begin{pmatrix}
1&1\\
1&-1
\end{pmatrix}
{\ }_{\begin{pmatrix}
a&b\\
c&d
\end{pmatrix}}\!\otimes
\begin{pmatrix}
1&1\\
1&-1
\end{pmatrix}=
\begin{pmatrix}
a&a&b&b\\
c&-c&d&-d\\ 
a&a&-b&-b\\
c&-c&-d&d
\end{pmatrix}$$

With the same data in hand, we can form as well the following matrix:
$$F_2\otimes_QF_2=\begin{pmatrix}
1&1\\
1&-1
\end{pmatrix}
\otimes_{\begin{pmatrix}
a&b\\
c&d
\end{pmatrix}}
\begin{pmatrix}
1&1\\
1&-1
\end{pmatrix}=
\begin{pmatrix}
a&b&a&b\\
a&-b&a&-b\\ 
c&d&-c&-d\\
c&-d&-c&d
\end{pmatrix}$$

Observe that the above two matrices are indeed Hadamard. In fact, these matrices appear as particular cases of the following general construction, due to Di\c t\u a \cite{dit}:

\begin{proposition}
If $H\in M_N(\mathbb C)$ and $K\in M_M(\mathbb C)$ are Hadamard, then so are:
\begin{enumerate}
\item $H\!{\ }_Q\!\otimes K=(Q_{aj}H_{ij}K_{ab})_{ia,jb}$, with $Q\in M_{M\times N}(\mathbb T)$.

\item $H\otimes_QK=(Q_{ib}H_{ij}K_{ab})_{ia,jb}$, with $Q\in M_{N\times M}(\mathbb T)$.
\end{enumerate}
These two matrices will be called left and right Di\c t\u a deformations of $H\otimes K$.
\end{proposition}

\begin{proof}
If we denote by $R_{ia}^1$ the $ia$-th row of $H\!{\ }_Q\!\otimes K$, we have:
$$\langle R_{ia}^1,R_{kc}^1\rangle=\sum_{jb}Q_{aj}H_{ij}K_{ab}\cdot\bar{Q}_{cj}\bar{H}_{kj}\bar{K}_{cb}=M\delta_{ac}\sum_jH_{ij}\bar{H}_{kj}=MN\delta_{ac}\delta_{ik}$$

Also, if we denote by $R_{ia}^2$ the $ia$-th row of $H\otimes_QK$, we have:
$$\langle R_{ia}^2,R_{kc}^2\rangle=\sum_{jb}Q_{ib}H_{ij}K_{ab}\cdot\bar{Q}_{kb}\bar{H}_{kj}\bar{K}_{cb}=N\delta_{ik}\sum_bK_{ab}\bar{K}_{cb}=NM\delta_{ik}\delta_{ac}$$

Thus in both cases we have $\langle R_{ia},R_{kc}\rangle=NM\delta_{ia,kc}$, which gives the result.
\end{proof}

The left and right Di\c t\u a deformations are related as follows:

\begin{proposition}
We have an equivalence $H\!{\ }_Q\!\otimes K\simeq K\otimes_QH$.
\end{proposition}

\begin{proof}
According to the formulae in Proposition 1.4 above, we have:
$$(H\!{\ }_Q\!\otimes K)_{ia,jb}=q_{aj}H_{ij}K_{ab}=q_{aj}K_{ab}H_{ij}=(K\otimes_QH)_{ai,bj}$$

Now since the transformation $M_{ia,jb}\to M_{ai,bj}$ is implemented by certain permutations of the rows and columns, the above two matrices are indeed equivalent.
\end{proof}

Observe now that, if we look at the complex Hadamard matrices modulo the equivalence relation in Definition 1.3, we can always assume that our parameter matrix $Q$ is ``dephased'', in the sense that its first row and column consist of 1 entries only.

As an illustration here, at $N=M=2$ we have the following result:

\begin{proposition}
Any Di\c t\u a deformation of $F_{2,2}=F_2\otimes F_2$ is equivalent to
$$F_{2,2}^q
=\begin{pmatrix}
1&1&1&1\\
1&1&-1&-1\\
1&-1&q&-q\\ 
1&-1&-q&q
\end{pmatrix}$$
for a certain value of the parameter $q\in\mathbb T$.
\end{proposition}

\begin{proof}
First, by using Proposition 1.5 we may restrict attention to the case of right Di\c t\u a deformations. But here, with $Q=(^a_b{\ }^c_d)$, our claim is that we have $F_2\otimes_QF_2\simeq F_{2,2}^q$, with $q=ad/bc$. Indeed, by dephasing our matrix we obtain:
$$F_2\otimes_QF_2=\begin{pmatrix}
a&b&a&b\\
a&-b&a&-b\\ 
c&d&-c&-d\\
c&-d&-c&d
\end{pmatrix}
\simeq\begin{pmatrix}
1&1&1&1\\
1&-1&1&-1\\ 
\frac{c}{a}&\frac{d}{b}&-\frac{c}{a}&-\frac{d}{b}\\
\frac{c}{a}&-\frac{d}{b}&-\frac{c}{a}&\frac{d}{b}
\end{pmatrix}\simeq
\begin{pmatrix}
1&1&1&1\\
1&-1&1&-1\\
1&\frac{ad}{bc}&-1&-\frac{ad}{bc}\\
1&-\frac{ad}{bc}&-1&\frac{ad}{bc}
\end{pmatrix}$$

Now by interchanging the middle columns, we obtain the matrix $F_{2,2}^q$, as claimed.
\end{proof}

Observe that, in view of Proposition 1.6, Haagerup's result in \cite{haa} tells us that all complex Hadamard matrices of order $N\leq 5$ appear as Di\c t\u a deformations of Fourier matrices.

Back now to the general case, we have the following key definition:

\begin{definition}
A complex Hadamard matrix $H\in M_N(\mathbb C)$ is called:
\begin{enumerate}
\item Of Buston type, if all its entries are roots of unity, of a given order $s<\infty$.

\item Regular, if all scalar products between rows decompose as sums of cycles.
\end{enumerate}
\end{definition}

Here by ``cycle'' we mean a full sum of roots of unity, possibly rotated by a scalar:
$$C_{n,\lambda}=\sum_{k=1}^n\lambda e^{2k\pi i/n}$$

As a basic example, all the Fourier matrices are of Butson type, and their Di\c t\u a deformations are regular. One interest in the above notions comes from the fact that the regular complex Hadamard matrices can be fully classified up to $N=6$. See \cite{bbs}.

In general, the regularity condition is Definition 1.7 is quite poorly understood:

\begin{conjecture}[Regularity]
Any Butson matrix is regular.
\end{conjecture}

More precisely, for exponents of type $s=p^a$ or $s=p^aq^b$ one can prove that any vanishing sum of $s$-roots of unity decomposes a sum of cycles, but this is not true in general. See \cite{lle}. Here is actually the simplest counterexample, with $w=e^{2\pi i/30}$:
$$w^5+w^6+w^{12}+w^{18}+w^{24}+w^{25}=0$$

So, what Conjecture 1.8 says is that such a ``tricky sum'' cannot be used for constructing a complex Hadamard matrix. We refer to \cite{bbs} for further details on this conjecture.

Let $C_N$ be the real algebraic manifold formed by the $N\times N$ complex Hadamard matrices. We denote by $M_x$ an unspecified neighborhood of a point in a manifold, $x\in M$. Also, for $q\in\mathbb T_1$, meaning that $q\in\mathbb T$ is close to $1$, we define $q^r$ with $r\in\mathbb R$ by $(e^{it})^r=e^{itr}$.

\begin{proposition}
For $H\in C_N$ and $A\in M_N(\mathbb R)$, the following are equivalent:
\begin{enumerate}
\item $H_{ij}^q=H_{ij}q^{A_{ij}}$ is an Hadamard matrix, for any $q\in\mathbb T_1$.

\item $\sum_kH_{ik}\bar{H}_{jk}q^{A_{ik}-A_{jk}}=0$, for any $i\neq j$ and any $q\in\mathbb T_1$.

\item $\sum_kH_{ik}\bar{H}_{jk}\varphi(A_{ik}-A_{jk})=0$, for any $i\neq j$ and any $\varphi:\mathbb R\to\mathbb C$.

\item $\sum_{k\in E_{ij}^r}H_{ik}\bar{H}_{jk}=0$ for any $i\neq j$ and $r\in\mathbb R$, where $E_{ij}^r=\{k|A_{ik}-A_{jk}=r\}$.
\end{enumerate}
\end{proposition}

\begin{proof}
All the equivalences are elementary, and can be proved as follows:

$(1)\iff(2)$ Indeed, the scalar products between the rows of $H^q$ are:
$$<H^q_i,H^q_j>=\sum_kH_{ik}q^{A_{ik}}\bar{H}_{jk}\bar{q}^{A_{jk}}=\sum_kH_{ik}\bar{H}_{jk}q^{A_{ik}-A_{jk}}$$

$(2)\implies(4)$ This follows from the following formula, and from the fact that the power functions $\{q^r|r\in\mathbb R\}$ over the unit circle $\mathbb T$ are linearly independent:
$$\sum_kH_{ik}\bar{H}_{jk}q^{A_{ik}-A_{jk}}=\sum_{r\in\mathbb R}q^r\sum_{k\in E_{ij}^r}H_{ik}\bar{H}_{jk}$$

$(4)\implies(3)$ This follows from the following formula:
$$\sum_kH_{ik}\bar{H}_{jk}\varphi(A_{ik}-A_{jk})=\sum_{r\in\mathbb R}\varphi(r)\sum_{k\in E_{ij}^r}H_{ik}\bar{H}_{jk}$$

$(3)\implies(2)$ This simply follows by taking $\varphi(r)=q^r$.
\end{proof}

Observe that in the above statement the condition (4) is purely combinatorial.

In order to understand the above type of deformations, it is convenient to enlarge attention to all types of deformations. We keep using the neighborhood notation $M_x$ introduced above, and we consider functions of type $f:M_x\to N_y$, which by definition satisfy $f(x)=y$. With these conventions, we introduce the following notions:

\begin{definition}
Let $H\in M_N(\mathbb C)$ be a complex Hadamard matrix.
\begin{enumerate}
\item A deformation of $H$ is a smooth function $f:\mathbb T_1\to (C_N)_H$.

\item The deformation is called ``affine'' if $f_{ij}(q)=H_{ij}q^{A_{ij}}$, with $A\in M_N(\mathbb R)$.

\item We call ``trivial'' the deformations of type $f_{ij}(q)=H_{ij}q^{a_i+b_j}$, with $a,b\in\mathbb R^N$.
\end{enumerate}
\end{definition}

Here the adjective ``affine'' comes from $f_{ij}(e^{it})=H_{ij}e^{iA_{ij}t}$, because the function $t\to A_{ij}t$ which produces the exponent is indeed affine. As for the adjective ``trivial'', this comes from the fact that $f(q)=(H_{ij}q^{a_i+b_j})_{ij}$ is obtained from $H$ by multiplying the rows and columns by certain numbers in $\mathbb T$, so it is automatically Hadamard. See \cite{tz1}.

The basic example of an affine deformation comes from the Di\c t\u a deformations $H\otimes_QK$, by taking all parameters $q_{ij}\in\mathbb T$ to be powers of $q\in\mathbb T$. As an example, here are the exponent matrices coming from the left and right Di\c t\u a deformations of $F_2\otimes F_2$:
$$A_l=
\begin{pmatrix}
a&a&b&b\\
c&c&d&d\\ 
a&a&b&b\\
c&c&d&d
\end{pmatrix}\quad\quad\quad
A_r=
\begin{pmatrix}
a&b&a&b\\
a&b&a&b\\ 
c&d&c&d\\
c&d&c&d
\end{pmatrix}$$

In order to investigate the above types of deformations, we will use the corresponding tangent vectors. So, let us first recall that the manifold $C_N$ is given by:
$$C_N=M_N(\mathbb T)\cap\sqrt{N}U_N$$

This observation leads to the following definitions, where in the first part we denote by $T_xM$ the tangent space to a point in a smooth manifold, $x\in M$:

\begin{definition}
Associated to a point $H\in C_N$ are the following objects:
\begin{enumerate}
\item The enveloping tangent space: $\widetilde{T}_HC_N=T_HM_N(\mathbb T)\cap T_H\sqrt{N}U_N$.

\item The tangent cone $T_HC_N$: the set of tangent vectors to the deformations of $H$.

\item The affine tangent cone $T_H^\circ C_N$: same as above, using affine deformations only.

\item The trivial tangent cone $T_H^\times C_N$: as above, using trivial deformations only.
\end{enumerate}
\end{definition}

Observe that $\widetilde{T}_HC_N,T_H^\times C_N$ are real linear spaces, and that $T_HC_N,T_H^\circ C_N$ are two-sided cones, in the sense that they satisfy $\lambda\in\mathbb R,A\in T\implies\lambda A\in T$.

Observe also that we have inclusions of cones, as follows:
$$T_H^\times C_N\subset T_H^\circ C_N\subset T_HC_N\subset\widetilde{T}_HC_N$$

In more algebraic terms now, the above tangent cones are best described by the corresponding matrices, as follows:

\begin{theorem}
The cones $T_H^\times C_N\subset T_H^\circ C_N\subset T_HC_N\subset\widetilde{T}_HC_N$ are as follows:
\begin{enumerate}
\item $\widetilde{T}_HC_N$ can be identified with the linear space formed by the matrices $A\in M_N(\mathbb R)$ satisfying $\sum_kH_{ik}\bar{H}_{jk}(A_{ik}-A_{jk})=0$, for any $i,j$.

\item $T_HC_N$ consists of those matrices $A\in M_N(\mathbb R)$ appearing as $A_{ij}=g_{ij}'(0)$, where $g:M_N(\mathbb R)_0\to M_N(\mathbb R)_0$ satisfies $\sum_kH_{ik}\bar{H}_{jk}e^{i(g_{ik}(t)-g_{jk}(t))}=0$ for any $i,j$.

\item $T^\circ_HC_N$ is formed by the matrices $A\in M_N(\mathbb R)$ satisfying $\sum_kH_{ik}\bar{H}_{jk}q^{A_{ik}-A_{jk}}=0$, for any $i\neq j$ and any $q\in\mathbb T$.

\item $T^\times_HC_N$ is formed by the matrices $A\in M_N(\mathbb R)$ which are of the form $A_{ij}=a_i+b_j$, for certain vectors $a,b\in\mathbb R^N$.
\end{enumerate}
\end{theorem}

\begin{proof}
All these assertions can be deduced by using basic differential geometry:

(1) This result is from \cite{ban}, the idea being as follows. First, $M_N(\mathbb T)$ is defined by the algebraic relations $|H_{ij}|^2=1$, and with $H_{ij}=X_{ij}+iY_{ij}$ we have:
$$d|H_{ij}|^2=d(X_{ij}^2+Y_{ij}^2)=2(X_{ij}\dot{X}_{ij}+Y_{ij}\dot{Y}_{ij})$$

Now since an arbitrary vector $\xi\in T_HM_N(\mathbb C)$, written as $\xi=\sum_{ij}\alpha_{ij}\dot{X}_{ij}+\beta_{ij}\dot{Y}_{ij}$, belongs to $T_HM_N(\mathbb T)$ if and only if $\langle\xi,d|H_{ij}|^2\rangle=0$ for any $i,j$, we obtain:
$$T_HM_N(\mathbb T)=\left\{\sum_{ij}A_{ij}(Y_{ij}\dot{X}_{ij}-X_{ij}\dot{Y}_{ij})\Big|A_{ij}\in\mathbb R\right\}$$

We also know that $\sqrt{N}U_N$ is defined by the algebraic relations $\langle H_i,H_j\rangle=N\delta_{ij}$, where $H_1,\ldots,H_N$ are the rows of $H$. The relations $\langle H_i,H_i\rangle=N$ being automatic for the matrices $H\in M_N(\mathbb T)$, if for $i\neq j$ we let $L_{ij}=\langle H_i,H_j\rangle$, then we have:
$$\widetilde{T}_HC_N=\left\{\xi\in T_HM_N(\mathbb T)|\langle\xi,\dot{L}_{ij}\rangle=0,\,\forall i\neq j\right\}$$

On the other hand, differentiating the formula of $L_{ij}$ gives:
$$\dot{L}_{ij}=\sum_k(X_{ik}+iY_{ik})(\dot{X}_{jk}-i\dot{Y}_{jk})+(X_{jk}-iY_{jk})(\dot{X}_{ik}+i\dot{Y}_{ik})$$

Now if we pick $\xi\in T_HM_N(\mathbb T)$, written as above in terms of $A\in M_N(\mathbb R)$, we obtain:
$$\langle\xi,\dot{L}_{ij}\rangle=i\sum_k\bar{H}_{ik}H_{jk}(A_{ik}-A_{jk})$$

Thus we have reached to the description of $\widetilde{T}_HC_N$ in the statement. 

(2) Pick an arbitrary deformation, and write it as $f_{ij}(e^{it})=H_{ij}e^{ig_{ij}(t)}$. Observe first that the Hadamard condition corresponds to the equations in the statement, namely:
$$\sum_kH_{ik}\bar{H}_{jk}e^{i(g_{ik}(t)-g_{jk}(t))}=0$$

Observe also that by differentiating this formula at $t=0$, we obtain:
$$\sum_kH_{ik}\bar{H}_{jk}(g_{ik}'(0)-g_{jk}'(0))=0$$

Thus the matrix $A_{ij}=g_{ij}'(0)$ belongs indeed to $\widetilde{T}_HC_N$, so we obtain in this way a certain map $T_HC_N\to\widetilde{T}_HC_N$. In order to check that this map is indeed the correct one, we have to verify that, for any $i,j$, the tangent vector to our deformation is given by:
$$\xi_{ij}=g_{ij}'(0)(Y_{ij}\dot{X}_{ij}-X_{ij}\dot{Y}_{ij})$$

But this latter verification is just a one-variable problem. So, by dropping all $i,j$ indices (which is the same as assuming $N=1$), we have to check that for any point $H\in\mathbb T$, written $H=X+iY$, the tangent vector to the deformation $f(e^{it})=He^{ig(t)}$ is:
$$\xi=g'(0)(Y\dot{X}-X\dot{Y})$$

But this is clear, because the unit tangent vector at $H\in\mathbb T$ is $\eta=-i(Y\dot{X}-X\dot{Y})$, and  its coefficient coming from the deformation is $(e^{ig(t)})'_{|t=0}=-ig'(0)$. 

(3) Observe first that by taking the derivative at $q=1$ of the condition (2) in Proposition 1.9, of just by using the condition (3) there with the function $\varphi(r)=r$, we get:
$$\sum_kH_{ik}\bar{H}_{jk}\varphi(A_{ik}-A_{jk})=0$$

Thus we have a map $T_H^\circ C_N\to\widetilde{T}_HC_N$, and the fact that is map is indeed the correct one comes for instance from the computation in (2), with $g_{ij}(t)=A_{ij}t$.

(4) Observe first that the Hadamard matrix condition is satisfied:
$$\sum_kH_{ik}\bar{H}_{jk}q^{A_{ik}-A_{jk}}
=q^{a_i-a_j}\sum_kH_{ik}\bar{H}_{jk}
=\delta_{ij}$$

As for the fact that $T_H^\times C_N$ is indeed the space in the statement, this is clear.
\end{proof}

Let now $D_N\subset C_N$ be the real algebraic manifold formed by all the dephased $N\times N$ complex Hadamard matrices. Observe that we have a quotient map $C_N\to D_N$, obtained by dephasing. With this notation, we have the following refinement of (4) above:

\begin{proposition}
We have a direct sum decomposition of cones
$$T_H^\circ C_N=T_H^\times C_N\oplus T_H^\circ D_N$$
where at right we have the affine tangent cone to the dephased manifold $C_N\to D_N$.
\end{proposition}

\begin{proof}
As explained in \cite{ban}, if we denote by $M_N^\circ(\mathbb R)$ the set of matrices having $0$ outside the first row and column, we have a direct sum decomposition, as follows:
$$\widetilde{T}_H^\circ C_N=M_N^\circ(\mathbb R)\oplus\widetilde{T}_H^\circ D_N$$

Now by looking at the affine cones, and using Theorem 1.12, this gives the result.
\end{proof}

\section{The defect, revisited}

The following key definition, whose origins go back to the work of Karabegov \cite{kar} and Nicoara \cite{nic}, was given by Tadej and \.Zyczkowski in \cite{tz2}:

\begin{definition}
The undephased defect of a complex Hadamard matrix $H\in C_N$ is the real dimension $d(H)$ of the enveloping tangent space $\widetilde{T}_HC_N=T_HM_N(\mathbb T)\cap T_H\sqrt{N}U_N$.
\end{definition}

In view of Proposition 1.13, it is sometimes convenient to replace $d(H)$ by the related quantity $d'(H)=d(H)-2N+1$, called dephased defect of $H$. See \cite{bbe}, \cite{ta2}, \cite{tz2}. In what follows we will rather use $d(H)$ as defined above, and simply call it ``defect'' of $H$.

Here are a few basic properties of the defect:

\begin{proposition}
Let $H\in C_N$ be a complex Hadamard matrix.
\begin{enumerate}
\item If $H\simeq\widetilde{H}$ then $d(H)=d(\widetilde{H})$.

\item We have $2N-1\leq d(H)\leq N^2$.

\item If $d(H)=2N-1$, the image of $H$ in the dephased manifold $C_N\to D_N$ is isolated.
\end{enumerate}
\end{proposition}

\begin{proof}
All these results are well-known, the proof being as follows:

(1) If we let $K_{ij}=a_ib_jH_{ij}$ with $|a_i|=|b_j|=1$ be a trivial deformation of our matrix $H$, the equations for the enveloping tangent space for $K$ are:
$$\sum_ka_ib_kH_{ik}\bar{a}_j\bar{b}_k\bar{H}_{jk}(A_{ik}-A_{jk})=0$$

By simplifying we obtain the equations for $H$, so $d(H)$ is invariant under trivial deformations. Since $d(H)$ is invariant as well by permuting rows or columns, we are done.

(2) Consider the inclusions $T_H^\times C_N\subset T_HC_N\subset\widetilde{T}_HC_N$. Since $\dim(T_H^\times C_N)=2N-1$, the inequality at left holds indeed. As for the inequality at right, this is clear.

(3) If $d(H)=2N-1$ then $T_HC_N=T_H^\times C_N$, so any deformation of $H$ is trivial. Thus the image of $H$ in the quotient manifold $C_N\to D_N$ is indeed isolated, as stated.
\end{proof}

In the case of Fourier matrices, the computation of the defect is as follows:

\begin{proposition}
Let $F=F_G$ be the Fourier matrix of a group $G=\mathbb Z_{N_1}\times\ldots\mathbb Z_{N_k}$.
\begin{enumerate}
\item $\widetilde{T}_FC_N=\{PF^*|P_{ij}=P_{i+j,j}=\bar{P}_{i,-j}\}$.

\item $d(F)=\sum_{g\in G}[G:<g>]$.

\item $d(F)$ is also the number of $1$ entries of $F$.

\item For $G=\mathbb Z_N$ with $N=\prod_ip_i^{a_i}$ we have $d(F)=N\prod_i(1+a_i-\frac{a_i}{p_i})$.
\end{enumerate}
\end{proposition}

\begin{proof}
This is basically known from \cite{kar}, \cite{tz2}, with some improvements and generalizations coming from \cite{ban}, \cite{bbe}, \cite{ta2}, and the proof is as follows:

(1) According to Theorem 1.12 (1), the equations are $\sum_kF_{ik}\bar{F}_{jk}(A_{ik}-A_{jk})=0$. But these equations simply read $(AF)_{i,i-j}-(AF)_{j,i-j}=0$, and together with the fact that $A=(AF)F^*$ must be real, this gives the conditions in the statement. See \cite{ban}.

(2) The proof here uses an identification of real vector spaces, as follows:
$$\{P\in M_N(\mathbb C)|P_{ij}=P_{i+j,j}=\bar{P}_{i,-j}\}\simeq\bigoplus_{g\in G}C(G/<g>,\mathbb R)$$

Indeed, if we let $G_2=\{g\in G|2g=0\}$, and then choose a partition of type $G=G_2\sqcup X\sqcup(-X)$, the formula giving the above identification is $P=\oplus P_g$, with:
$$P_g(i)=
\begin{cases}
P_{ij}&(j\in G_2)\\
Re(P_{ij})&(j\in X)\\
Im(P_{ij})&(j\in -X)
\end{cases}$$

With this identification in hand, the result follows from (1).

(3) This observation, from \cite{kar}, follows from (2), and from the definition of $F$.

(4) This formula, due to Tadej and \.Zyczkowski, follows from (2) or (3). See \cite{tz2}.
\end{proof}

We discuss now the computation of the various tangent cones at a tensor product $H\otimes K$. The problem that we are interested in, raised by the work in \cite{ta1}, is to determine how the various tangent cones at $H,K$ glue together at $H\otimes K$. First, we have:

\begin{proposition}
We have $\widetilde{T}_HC_N\otimes\widetilde{T}_KC_M\subset\widetilde{T}_{H\otimes K}C_{NM}$.
\end{proposition}

\begin{proof}
Indeed, for a matrix $A=B\otimes C$, we have the following formulae:
\begin{eqnarray*}
\sum_{kc}(H\otimes K)_{ia,kc}\overline{(H\otimes K)}_{jb,kc}A_{ia,kc}
&=&\sum_kH_{ik}\bar{H}_{jk}B_{ik}\sum_cK_{ac}\bar{K}_{bc}C_{ac}\\
\sum_{kc}(H\otimes K)_{ia,kc}\overline{(H\otimes K)}_{jb,kc}A_{jb,kc}
&=&\sum_kH_{ik}\bar{H}_{jk}B_{jk}\sum_cK_{ac}\bar{K}_{bc}C_{bc}
\end{eqnarray*}

Now by assuming $B\in\widetilde{T}_HC_N$ and $C\in\widetilde{T}_KC_M$, the two quantities appearing above on the right are equal. Thus we have indeed $A\in\widetilde{T}_{H\otimes K}C_{NM}$, and we are done.
\end{proof}

Let us discuss now the computation of the various tangent cones for a Di\c t\u a deformation. Here we basically have just one result, when the deformation matrix is as follows:

\begin{definition}
A rectangular matrix $Q\in M_{N\times M}(\mathbb T)$ is called ``dephased and elsewhere generic'' if the entries on its first row and column are all equal to $1$, and the remaining $(N-1)(M-1)$ entries are algebrically independent over $\mathbb Q$.
\end{definition}

Here the last condition takes of course into account the fact that the entries of $Q$ themselves have modulus 1, the independence assumption being modulo this fact. 

We have the following result, extending the $4\times 4$ computations in \cite{ban}:

\begin{proposition}
If $H\in C_N,K\in C_M$ are dephased, of Butson type, and $Q\in M_{N\times M}(\mathbb T)$ is dephased and elsewhere generic, then $A=(A_{ia,kc})$ belongs to $\widetilde{T}_{H\otimes_QK}C_{NM}$ iff
$$A_{ac}^{ij}=A_{bc}^{ij},\quad A_{ac}^{ij}=\overline{A_{ac}^{ji}},\quad (A_{xy}^{ii})_{xy}\in\widetilde{T}_KC_M$$
hold for any $a,b,c$ and $i\neq j$, where $A_{ac}^{ij}=\sum_kH_{ik}\bar{H}_{jk}A_{ia,kc}$.
\end{proposition}

\begin{proof}
Consider the system for the enveloping tangent space, namely:
$$\sum_{kc}(H\otimes_QK)_{ia,kc}\overline{(H\otimes_QK)}_{jb,kc}(A_{ia,kc}-A_{jb,kc})=0$$

We have $(H\otimes_QK)_{ia,jb}=q_{ib}H_{ij}K_{ab}$, and so our system is:
$$\sum_cq_{ic}\bar{q}_{jc}K_{ac}\bar{K}_{bc}\sum_kH_{ik}\bar{H}_{jk}(A_{ia,kc}-A_{jb,kc})=0$$

Consider now the variables $A_{ac}^{ij}=\sum_kH_{ik}\bar{H}_{jk}A_{ia,kc}$ in the statement. We have:
$$\overline{A_{ac}^{ij}}=\sum_k\bar{H}_{ik}H_{jk}A_{ia,kc}=\sum_kH_{jk}\bar{H}_{ik}A_{ia,kc}$$

Thus, in terms of these variables, our system becomes simply:
$$\sum_cq_{ic}\bar{q}_{jc}K_{ac}\bar{K}_{bc}(A_{ac}^{ij}-\overline{A_{bc}^{ji}})=0$$

More precisely, the above equations must hold for any $i,j,a,b$. By distinguishing now two cases, depending on whether $i,j$ are equal or not, the situation is as follows:

(1) Case $i\neq j$. In this case, let us look at the row vector of parameters, namely:
$$(q_{ic}\bar{q}_{jc})_c=(1,q_{i1}\bar{q}_{j1},\ldots,q_{iN}\bar{q}_{jN})$$

Now since $Q$ was assumed to be dephased and elsewhere generic, and because of our assumption $i\neq j$, the entries of the above vector are linearly independent over $\bar{\mathbb Q}$. But, since by linear algebra we can restrict attention to the computation of the solutions over $\bar{\mathbb Q}$, the $i\neq j$ part of our system simply becomes $A_{ac}^{ij}=\overline{A_{bc}^{ji}}$, for any $a,b,c$ and any $i\neq j$. Now by making now $a,b,c$ vary, we are led to the following equations:
$$A_{ac}^{ij}=A_{bc}^{ij},\quad A_{ac}^{ij}=\overline{A_{ac}^{ji}},\quad\forall a,b,c,i\neq j$$

(2) Case $i=j$. In this case the parameters cancel, and our equations become:
$$\sum_cK_{ac}\bar{K}_{bc}(A_{ac}^{ii}-\overline{A_{bc}^{ii}})=0,\quad\forall a,b,c,i$$

On the other hand, we have $A_{ac}^{ii}=\sum_kA_{ia,kc}$, and so our equations become:
$$\sum_cK_{ac}\bar{K}_{bc}(A_{ac}^{ii}-A_{bc}^{ii})=0,\quad\forall a,b,c,i$$

But these are precisely the equations for the space $\widetilde{T}_KC_M$, and we are done.
\end{proof}

Let us go back now to the usual tensor product situation, and look at the affine cones. The problem here is that of finding the biggest subcone of $T_{H\otimes K}^\circ C_{NM}$, obtained by gluing $T_H^\circ C_N,T_K^\circ C_M$. Our answer here, which takes into account the two ``semi-trivial'' cones coming from the left and right Di\c t\u a deformations, is as follows:

\begin{theorem}
The cones $T_H^\circ C_N=\{B\}$ and $T_K^\circ C_M=\{C\}$ glue via the formulae
$$A_{ia,jb}=\lambda B_{ij}+\psi_jC_{ab}+X_{ia}+Y_{jb}+F_{aj}$$
$$A_{ia,jb}=\phi_bB_{ij}+\mu C_{ab}+X_{ia}+Y_{jb}+E_{ib}$$
producing in this way two subcones of the affine cone $T_{H\otimes K}^\circ C_{NM}=\{A\}$.
\end{theorem}

\begin{proof}
Indeed, the idea is that $X_{ia},Y_{jb}$ are the trivial parameters, and that $E_{ib},F_{aj}$ are the Di\c t\u a parameters. In order to prove the result, we use the criterion in Theorem 1.12 (3) above. So, given a matrix $A=(A_{ia,jb})$, consider the following quantity:
$$P=\sum_{kc}H_{ik}\bar{H}_{jk}K_{ac}\bar{K}_{bc}q^{A_{ia,kc}-A_{jb,kc}}$$

Let us prove now the first statement, namely that for any choice of matrices $B\in T_H^\circ C_N,C\in T_H^\circ C_M$ and of parameters $\lambda,\psi_j,X_{ia},Y_{jb},F_{aj}$, the first matrix $A=(A_{ia,jb})$ constructed in the statement belongs indeed to $T_{H\otimes K}^\circ C_{NM}$. We have: 
$$A_{ia,kc}=\lambda B_{ik}+\psi_kC_{ac}+X_{ia}+Y_{kc}+F_{ak}$$
$$A_{jb,kc}=\lambda B_{jk}+\psi_kC_{bc}+X_{jb}+Y_{kc}+F_{bk}$$

Now by substracting, we obtain:
$$A_{ia,kc}-A_{jb,kc}=\lambda(B_{ik}-B_{jk})+\psi_k(C_{ac}-C_{bc})+(X_{ia}-X_{jb})+(F_{ak}-F_{bk})$$

It follows that the above quantity $P$ is given by:
\begin{eqnarray*}
P
&=&\sum_{kc}H_{ik}\bar{H}_{jk}K_{ac}\bar{K}_{bc}q^{\lambda(B_{ik}-B_{jk})+\psi_k(C_{ac}-C_{bc})+(X_{ia}-X_{jb})+(F_{ak}-F_{bk})}\\
&=&q^{X_{ia}-X_{jb}}\sum_kH_{ik}\bar{H}_{jk}q^{F_{ak}-F_{bk}}q^{\lambda(B_{ik}-B_{jk})}\sum_cK_{ac}\bar{K}_{bc}(q^{\psi_k})^{C_{ac}-C_{bc}}\\
&=&\delta_{ab}q^{X_{ia}-X_{ja}}\sum_kH_{ik}\bar{H}_{jk}(q^\lambda)^{B_{ik}-B_{jk}}=\delta_{ab}\delta_{ij}
\end{eqnarray*}

Thus Theorem 1.12 (3) applies and tells us that we have $A\in T_{H\otimes K}^\circ C_{NM}$, as claimed. In the second case now, the proof is similar. First, we have:
$$A_{ia,kc}=\phi_cB_{ik}+\mu C_{ac}+X_{ia}+Y_{kc}+E_{ic}$$
$$A_{jb,kc}=\phi_cB_{jk}+\mu C_{bc}+X_{jb}+Y_{kc}+E_{jc}$$

Thus by substracting, we obtain:
$$A_{ia,kc}-A_{jb,kc}=\phi_c(B_{ik}-B_{jk})+\mu(C_{ac}-C_{bc})+(X_{ia}-X_{jb})+(E_{ic}-E_{jc})$$

It follows that the above quantity $P$ is given by:
\begin{eqnarray*}
P
&=&\sum_{kc}H_{ik}\bar{H}_{jk}K_{ac}\bar{K}_{bc}q^{\phi_c(B_{ik}-B_{jk})+\mu(C_{ac}-C_{bc})+(X_{ia}-X_{jb})+(E_{ic}-E_{jc})}\\
&=&q^{X_{ia}-X_{jb}}\sum_cK_{ac}\bar{K}_{bc}q^{E_{ic}-E_{jc}}q^{\mu(C_{ac}-C_{bc})}\sum_kH_{ik}\bar{H}_{jk}(q^{\phi_c})^{B_{ik}-B_{jk}}\\
&=&\delta_{ij}q^{X_{ia}-X_{ib}}\sum_cK_{ac}\bar{K}_{bc}(q^\mu)^{C_{ac}-C_{bc}}=\delta_{ij}\delta_{ab}
\end{eqnarray*}

Thus Theorem 1.12 (3) applies again, and gives the result.
\end{proof}

We believe Theorem 2.7 above to be ``optimal'', in the sense that  nothing more can be said about the affine tangent space $T_{H\otimes K}^\circ C_{NM}$, in the general case. This claim is supported by various computations for $F_N\otimes F_M$, and by results in \cite{bbe}, \cite{tz2}.

Let us discuss now some rationality questions:

\begin{definition}
The rational defect of $H\in C_N$ is the following number:
$$d_\mathbb Q(H)=\dim_\mathbb Q(\widetilde{T}_HC_N\cap M_N(\mathbb Q))$$
The vector space on the right will be called rational enveloping tangent space at $H$.
\end{definition}

As a first observation, this notion can be extended to all the tangent cones at $H$, and by using an arbitrary field $\mathbb K\subset\mathbb C$ instead of $\mathbb Q$. Indeed, we can set:
$$T_H^\ast C_N(\mathbb K)=T_H^\ast C_N\cap M_N(\mathbb K)$$

However, in what follows we will be interested only in the objects constructed in Definition 2.8. It follows from definitions that $d_\mathbb Q(H)\leq d(H)$, and we have:

\begin{conjecture}[Rationality]
For the Butson matrices we have $d_\mathbb Q(H)=d(H)$.
\end{conjecture}

In fact, the original statement in \cite{ban} is that the above equality should hold in the regular case. However, since the regular case is not known to fully cover the Butson matrix case, as explained in Conjecture 1.8, we prefer to state our conjecture as above.

Let $C_N(s)$ be the set of $N\times N$ complex Hadamard matrices having as entries the $s$-th roots of unity. With this notation, we have the following elementary result:

\begin{proposition}
The rationality conjecture holds for $H\in C_N(s)$ with $s=2,3,4,6$.
\end{proposition}

\begin{proof}
Let us recall that the equations for the enveloping tangent space are:
$$\sum_kH_{ik}\bar{H}_{jk}(A_{ik}-A_{jk})=0$$

In the case $s=2$ these equations are all real, and have rational ($\pm 1$) coefficients. In the case $s=3,6$ we can use the well-known fact that, with $j=e^{2\pi i/3}$, the real solutions of $x+jy+j^2z=0$ are those satisfying $x=y=z$, and we conclude that once again our system, after some manipulations, is equivalent to a real system having rational coefficients. As for the case $s=4$, here the coefficients are $1,i,-1,-i$, so by taking the real and imaginary parts, we reach once again to system with rational coefficients.

Thus, in all cases under investigation, $s=2,3,4,6$, we have a real system with rational coefficients, and the result follows from standard linear algebra.
\end{proof}

Observe that the above method cannot work at $s=5$, where the equation $a+wb+w^2c+w^3d+w^4e=0$ with $w=e^{2\pi i/5}$ and $a,b,c,d,e\in\mathbb R$ can have ``exotic'' solutions.

We will prove in section 3 that Conjecture 2.9 is verified for the Fourier matrices.

Finally, let us go back to Proposition 2.3, and to the formula $d(H)=|1\in H|$ appearing there. We expect this formula to be valid under much more general assumptions:

\begin{problem}
When does the defect formula $d(H)=|1\in H|$ hold?
\end{problem}

We believe for instance that this formula should hold for matrices of type $H=F_G\otimes_QF_K$, under fairly general assumptions on the deformation matrix $Q\in M_{N\times M}(\mathbb T)$. For instance Proposition 2.6 above suggests that the result might hold for $Q$ dephased and elsewhere generic. However, we do not know what the exact assumptions on $Q$ should be.

\section{Fourier matrices}

In this section we state and prove our main result, concerning the enveloping tangent space at $F_N$. As a consequence, we will see that Conjecture 2.9 holds for $F_N$.

Let us first discuss in detail the isotypic case, following \cite{bbe}, \cite{tz1}. First, we have:

\begin{proposition}
For $F=F_p$ with $p$ prime, we have
$$T_F^\times C_p=T_F^\circ C_p=T_HC_p=\widetilde{T}_FC_p$$
and this space consists of the matrices of type $A_{ij}=a_i+b_j$, with $a,b\in\mathbb R^p$.
\end{proposition}

\begin{proof}
It is enough to check that we have $T_F^\times C_p=\widetilde{T}_FC_p$, which means that any matrix $A\in\widetilde{T}_FC_p$ should decompose as $A_{ij}=a_i+b_j$, with $a,b\in\mathbb R^p$. At $p=2$ we have:
$$A=PF^*=\begin{pmatrix}a&c\\ b&c\end{pmatrix}\begin{pmatrix}1&1\\ 1&-1\end{pmatrix}=\begin{pmatrix}a+c&a-c\\ b+c&b-c\end{pmatrix}$$

This gives the result. At $p=3$ now, with $w=e^{2\pi i/3}$ we have:
$$P=\begin{pmatrix}a&z&\bar{z}\\ b&z&\bar{z}\\ c&z&\bar{z}\end{pmatrix},\quad\quad 
F^*=\begin{pmatrix}1&1&1\\ 1&w^2&w\\ 1&w&w^2\end{pmatrix}$$

Thus the exponent matrices have a similar look, and this gives the result:
$$A=PF^*=\begin{pmatrix}
a+z+\bar{z}&a+w^2z+w\bar{z}&a+wz+w^2\bar{z}\\
b+z+\bar{z}&b+w^2z+w\bar{z}&b+wz+w^2\bar{z}\\
c+z+\bar{z}&c+w^2z+w\bar{z}&c+wz+w^2\bar{z}
\end{pmatrix}$$

In general now, we have $A_{00}+A_{ij}=A_{i0}+A_{0j}$, and this gives the result.
\end{proof}

Observe that Proposition 3.1 is somehow not fully satisfactory, because the fact that the defect is $d(F)=2p-1$ is not entirely obvious from it. So, let us improve it:

\begin{proposition}
For $F=F_p$, the elements $A\in\widetilde{T}_FC_p$ are the solutions of
$$A_{ij}=L^{00}_{00}+L^{01}_{0j}+L^{10}_{i0}$$
where the $L$ variables are free, with $L^{01}_{0j}=0$ for $j\neq 0$, and $L^{10}_{i0}=0$ for $i\neq 0$.
\end{proposition}

\begin{proof}
The claim is that there exist free variables $L$, such that:
\begin{eqnarray*}
A_{00}&=&L^{00}_{00}\\
A_{0j}&=&L^{00}_{00}+L^{01}_{0j}\quad (j\neq 0)\\
A_{i0}&=&L^{00}_{00}+L^{10}_{i0}\quad (i\neq 0)\\
A_{ij}&=&L^{00}_{00}+L^{01}_{0j}+L^{10}_{i0}\quad (i,j\neq 0)
\end{eqnarray*}

But this follows from Proposition 3.1, by taking $L^{00}_{00}=A_{00}$, and so on.
\end{proof}

Let us discuss now the case $N=p^2$. We recall that $M_n^\circ(\mathbb R)$ is by definition the set of real $n\times n$ matrices having $0$ on the first row and column. We will need the following key result, regarding the enveloping tangent space to the dephased manifold $C_N\to D_N$:

\begin{proposition}
For $F=F_{p^2}$ with $p$ prime, we have
$$\widetilde{T}_FD_N=\{A\in M_N(\mathbb R)|\exists M\in M_p^\circ(\mathbb R),\,A_{ij}=M_{\bar{i},\bar{j}}\}$$
where $\bar{i},\bar{j}$ are the reminders of $i,j$ modulo $p$.
\end{proposition}

\begin{proof}
Let us first work out the case $p=2$. Here the defect is $8$, and:
$$P=\begin{pmatrix}a&z&e&\bar{z}\\ b&z&f&\bar{z}\\ c&z&e&\bar{z}\\ d&z&f&\bar{z}\end{pmatrix},\quad\quad F^*=\begin{pmatrix}1&1&1&1\\ 1&-i&-1&i\\ 1&-1&1&-1\\ 1&i&-1&-i\end{pmatrix}$$

Thus with $z=p+iq$ we obtain the following formula:
$$A=PF^*=\begin{pmatrix}
a+e+2p&a-e+2q&a+e-2p&a-e-2q\\
b+f+2p&b-f+2q&b+f-2p&b-f-2q\\
c+e+2p&c-e+2q&c+e-2p&c-e-2q\\
d+f+2p&d-f+2q&d+f-2p&d-f-2q
\end{pmatrix}$$

Now by assuming that $A$ is dephased, as in the statement, we obtain that we have $a=c=e=p=q=0$ and $b=d=-f$, and so our matrix is given by:
$$A=\begin{pmatrix}
0&0&0&0\\
0&-2f&0&-2f\\
0&0&0&0\\
0&-2f&0&-2f
\end{pmatrix}$$

Thus we reach to the conclusion in the statement. At $p=3$ now, we have:
$$A=\begin{pmatrix}
a&x&y&u&z&\bar{z}&\bar{u}&\bar{y}&\bar{x}\\
b&x&y&v&z&\bar{z}&\bar{v}&\bar{y}&\bar{x}\\
c&x&y&w&z&\bar{z}&\bar{w}&\bar{y}&\bar{x}\\
d&x&y&u&z&\bar{z}&\bar{u}&\bar{y}&\bar{x}\\
e&x&y&v&z&\bar{z}&\bar{v}&\bar{y}&\bar{x}\\
f&x&y&w&z&\bar{z}&\bar{w}&\bar{y}&\bar{x}\\
g&x&y&u&z&\bar{z}&\bar{u}&\bar{y}&\bar{x}\\
h&x&y&v&z&\bar{z}&\bar{v}&\bar{y}&\bar{x}\\
i&x&y&w&z&\bar{z}&\bar{w}&\bar{y}&\bar{x}
\end{pmatrix}
\begin{pmatrix}
1&1&1&1&1&1&1&1&1\\
1&w^8&w^7&w^6&w^5&w^4&w^3&w^2&w\\
1&w^7&w^5&w^3&w&w^8&w^6&w^4&w^2\\
1&w^6&w^3&1&w^6&w^3&1&w^6&w^3\\
1&w^5&w&w^6&w^2&w^7&w^3&w^8&w^4\\
1&w^4&w^8&w^3&w^7&w^2&w^6&w&w^5\\
1&w^3&w^6&1&w^3&w^6&1&w^3&w^6\\
1&w^2&w^4&w^6&w^8&w&w^3&w^5&w^7\\
1&w&w^2&w^3&w^4&w^5&w^6&w^7&w^8
\end{pmatrix}$$

By assuming now that $A$ is dephased, we obtain the following equations:
$$a=x=y=z=u=0,\quad d=g=0$$
$$b=e=h=-(v+\bar{v}),\quad c=f=i=-(w+\bar{w})$$

Now set $v=\alpha+i\beta$, $w=\gamma+i\delta$, and consider the following variables:
$$x=-3\alpha+(iw^6-iw^3)\beta,\quad y=-3\alpha+(iw^3-iw^6)\beta$$
$$z=-3\gamma+(iw^6-iw^3)\delta,\quad t=-3\gamma+(iw^3-iw^6)\delta$$

Observe that we have in fact $iw^6-iw^3=\frac{\sqrt{3}}{2}$. In terms of $x,y,z,t$, we have:
$$A=\begin{pmatrix}
0&0&0&0&0&0&0&0&0\\
0&x&y&0&x&y&0&x&y\\
0&z&t&0&z&t&0&z&t\\
0&0&0&0&0&0&0&0&0\\
0&x&y&0&x&y&0&x&y\\
0&z&t&0&z&t&0&z&t\\
0&0&0&0&0&0&0&0&0\\
0&x&y&0&x&y&0&x&y\\
0&z&t&0&z&t&0&z&t
\end{pmatrix}$$

Thus we have reached again to the conclusion in the statement. The general case is similar, and we refer here to the proof of Lemma 3.5 below, at $a=2$.
\end{proof}

Let us add now the trivial deformation part, and reformulate the result a bit as in Proposition 3.2 above. The statement that we obtain is as follows:

\begin{proposition}
For $F=F_{p^2}$, the elements $A\in\widetilde{T}_FC_N$ are the solutions of
$$A_{ij}=L^{00}_{00}+L^{01}_{0\bar{j}}+L^{10}_{\bar{i}0}+L^{02}_{0j}+L^{20}_{i0}+L^{11}_{\bar{i}\bar{j}}$$
where the $L$ variables are free, and dephased, and $\bar{i},\bar{j}$ are the reminders of $i,j$ modulo $p$.
\end{proposition}

\begin{proof}
By adding trivial deformations to the formula in Proposition 3.3, and then by reparametrizing these trivial deformations as in Proposition 3.2, we get:
$$A_{ij}=\alpha+a_i+b_j+M_{\bar{i}\bar{j}}$$

Here $a_i,b_j$ are free variables, with the conventions $a_i=0$ for $i\neq 0$, and $b_j=0$ for $j\neq 0$. Now by using the further splitting of indices modulo $p$, we can write:
$$A_{ij}=\alpha+\beta_{\bar{i}}+\gamma_{\bar{j}}+\delta_i+\varepsilon_j+M_{\bar{i}\bar{j}}$$

Here $\beta,\gamma,\delta,\varepsilon$ are now dephased, and this gives the formula in the statement.
\end{proof}

We are now in position of stating and proving our main technical result. We say that a matrix $L^{rs}$ over the group $\mathbb Z_{p^r}\times\mathbb Z_{p^s}$ is dephased if its nonzero entries belong to:
$$X_{rs}=(\mathbb Z_{p^r}-\mathbb Z_{p^{r-1}})\times(\mathbb Z_{p^s}-\mathbb Z_{p^{s-1}})$$

Here, and in what follows, we use the convention $\mathbb Z_{p^{-1}}=\emptyset$. 

\begin{lemma}
For $F=F_{p^a}$, the elements $A\in\widetilde{T}_FC_N$ are the solutions of
$$A_{ij}=\sum_{r+s\leq a}L^{rs}_{p^{a-r}i,p^{a-s}j}$$
where the $L$ variables are free, and form dephased matrices $L^{rs}$.
\end{lemma}

\begin{proof}
Observe first that at $a=1,2$ this follows from Proposition 3.2 and Proposition 3.4 respectively. Observe also that in the general case, the number of $L$ variables is:
\begin{eqnarray*}
d
&=&\sum_{r+s\leq a}|\mathbb Z_{p^r}-\mathbb Z_{p^{r-1}}|\cdot|\mathbb Z_{p^s}-\mathbb Z_{p^{s-1}}|=\sum_{r\leq a}p^{a-r}|\mathbb Z_{p^r}-\mathbb Z_{p^{r-1}}|\\
&=&p^a+\sum_{r=1}^ap^{a-r}(p^r-p^{r-1})=p^a+a(p-1)p^{a-1}=(p+ap-a)p^{a-1}
\end{eqnarray*}

Thus the number of $L$ variables equals the defect $d(F)$, so it is indeed the good one. As for the proof now, in the general case, this is quite similar to the one at $a=1,2$.

More precisely, consider the map $L\to A$. This map is linear, and in view of the above calculation, it is enough to prove that this map is injective, and has the correct target.

For the injectivity part, recall that at $a=2$ the formula in the statement reads:
$$A_{ij}=L^{00}_{00}+L^{01}_{0,pj}+L^{10}_{pi,0}+L^{02}_{0j}+L^{20}_{i0}+L^{11}_{pi,pj}$$

Now assume $A=0$. Then with $i=j=0$ we get $L^{00}_{00}=0$. Using this, with $i=0$ and $pj=0,j\neq 0$ we get $L^{00}_{00}+L^{02}_{0j}=0$, and so $L^{02}_{0j}=0$. So, with $i=0$ and $pj\neq 0$ we therefore obtain $L^{00}_{00}+L^{02}_{0j}+L^{01}_{0,pj}=0$, and so $L^{01}_{0,pj}=0$. Now the same method gives as well succesively $L^{20}_{i0}=0$ and $L^{10}_{pi,0}=0$, so we are left with $A_{ij}=L^{11}_{pi,pj}$, so we must have $L^{11}_{pi,pj}=0$ as well, and we are done. This method works of course for any $a\in\mathbb N$.

Regarding now the ``target'' part, we must prove $A\in\widetilde{T}_FC_N$. The equations are:
$$\sum_kw^{(i-j)k}\left(\sum_{r+s\leq a}L^{rs}_{p^{a-r}i,p^{a-s}k}-L^{rs}_{p^{a-r}j,p^{a-s}k}\right)=0$$

So, for any indices $i,j$ and any $r+s\leq a$, we must prove that we have:
$$\sum_kw^{(i-j)k}\left(L^{rs}_{p^{a-r}i,p^{a-s}k}-L^{rs}_{p^{a-r}j,p^{a-s}k}\right)=0$$

In order to do this, consider the following quantity:
$$X_{il}=\frac{1}{p^a}\sum_kw^{lk}L^{rs}_{p^{a-r}i,p^{a-s}k}$$

We must prove $X_{i,i-j}=X_{j,i-j}$. But, with $k=m+p^sn$, we have:
$$X_{il}
=\frac{1}{p^a}\sum_nw^{lp^sn}\sum_mw^{lm}L^{rs}_{p^{a-r}i,p^{a-s}m}
=\delta_{l0}\sum_mw^{lm}L^{rs}_{p^{a-r}i,p^{a-s}m}$$

Thus we have $l\neq 0\implies X_{il}=0$, and so $X_{i,i-j}=X_{j,i-j}$ and we are done.
\end{proof}

\begin{proposition}
For an isotypic Fourier matrix, $H=F_N$ with $N=p^a$, we have
$$T_H^\circ C_N=T_HC_N=\widetilde{T}_HC_N=\left\{A\in M_N(\mathbb R)\Big|A_{ij}=\sum_{r+s\leq a}L^{rs}_{p^{a-r}i,p^{a-s}j}\right\}$$
where the $L$ variables are free, and form dephased matrices $L^{rs}$. 
\end{proposition}

\begin{proof}
In view of Lemma 3.5, we just have to show that the defect of $F_N$ is exhausted by affine deformations. With $k=m+p^sn$, as in the proof of Lemma 3.5, we have:
\begin{eqnarray*}
\sum_kH_{ik}\bar{H}_{jk}q^{A_{ik}-A_{jk}}
&=&\sum_kw^{(i-j)k}\prod_{r+s\leq a}q^{L^{rs}_{p^{a-r}i,p^{a-s}k}-L^{rs}_{p^{a-r}j,p^{a-s}k}}\\
&=&\sum_nw^{(i-j)p^sn}\sum_mw^{(i-j)m}\prod_{r+s\leq a}q^{L^{rs}_{p^{a-r}i,p^{a-s}m}-L^{rs}_{p^{a-r}j,p^{a-s}m}}\\
&=&\delta_{ij}p^a\sum_mw^{(i-j)m}\prod_{r+s\leq a}q^{L^{rs}_{p^{a-r}i,p^{a-s}m}-L^{rs}_{p^{a-r}j,p^{a-s}m}}
\end{eqnarray*}

Now since this quantity vanishes for $i\neq j$, this gives the result.
\end{proof}

Observe that this result shows that Conjecture 2.9 holds for the isotypic Fourier matrices. We will see in what follows that the same happens for any Fourier matrix.

In order now to discuss the general case, $H=F_N$, we will need:

\begin{lemma}
If $G=H\times K$ is such that $(|H|,|K|)=1$, the canonical inclusion
$$\widetilde{T}_{F_H}C_{|H|}\otimes\widetilde{T}_{F_K}C_{|K|}\subset\widetilde{T}_{F_G}C_{|G|}$$
constructed in Proposition 2.4 above is an isomorphism.
\end{lemma}

\begin{proof}
We have $F_G=F_{H\times K}$, and the defect of this matrix is given by:
$$d(F_{H\times K})=\sum_{(h,k)\in H\times K}\frac{|H\times K|}{ord(h,k)}=\sum_{(h,k)\in H\times K}\frac{|H\times K|}{ord(h)ord(k)}=d(F_H)d(F_K)$$

Thus the inclusion in the statement must be indeed an isomorphism.
\end{proof}

With this lemma in hand, the idea now will be simply to ``glue'' the various isotypic formulae coming from Proposition 3.6. Indeed, let us recall from there that in the isotypic case, $N=p^a$, the parameter set for the enveloping tangent space is:
$$X(p^a)=\bigsqcup_{r+s\leq a}(\mathbb Z_{p^r}-\mathbb Z_{p^{r-1}})\times (\mathbb Z_{p^s}-\mathbb Z_{p^{s-1}})$$

Now since the defect is multiplicative over isotypic components, the parameter set in the general case, $N=p_1^{a_1}\ldots p_k^{a_k}$, will be simply given by:
$$X(p_1^{a_1}\ldots p_k^{a_k})=X(p_1^{a_1})\times\ldots\times X(p_k^{a_k})$$

One can obtain from this an even simpler description of the parameter set, just by expanding the product, and gluing the group components. Indeed:

\begin{definition}
Given a finite abelian group $G=\mathbb Z_{p_1^{r_1}}\times\ldots\times\mathbb Z_{p_k^{r_k}}$ we set:
$$G^\circ=(\mathbb Z_{p_1^{r_1}}-\mathbb Z_{p_1^{r_1-1}})\times\ldots\times(\mathbb Z_{p_k^{r_k}}-\mathbb Z_{p_k^{r_k-1}})$$
A matrix $L\in M_{G\times H}(\mathbb R)$ will be called dephased if $L_{ij}=0$ for any $(i,j)\not\in G^\circ\times H^\circ$.
\end{definition}

Observe now that, with the above notation $G^\circ$, the parameter set discussed above is given by the following simple formula:
$$X(N)=\bigsqcup_{G\times H\subset\mathbb Z_N}G^\circ\times H^\circ$$

In addition, we can see that the collection of dephased matrices $L\in M_{G\times H}(\mathbb R)$ , over all possible configurations $G\times H\subset\mathbb Z_N$, takes its parameters precisely in $X(N)$.

In order to formulate our main result, we will need one more definition:

\begin{definition}
Given $N=p_1^{a_1}\ldots p_k^{a_k}$ and a subgroup $G\subset \mathbb Z_N$, we set 
$$\varphi_G(i_1,\ldots,i_k)=(p_1^{a_1-r_1}i_1,\ldots p_k^{a_k-r_k}i_k)$$
where the exponents $r_i\leq a_i$ are given by $G=\mathbb Z_{p_1^{r_1}}\times\ldots\times\mathbb Z_{p_k^{r_k}}$.
\end{definition}

Observe that in the case $k=1$ this function is precisely the one appearing in Proposition 3.6 above. In fact, we have the following generalization of Proposition 3.6:

\begin{theorem}
For $H=F_N$ the vectors $A\in\widetilde{T}_HC_N$ appear as plain sums of type
$$A_{ij}=\sum_{G\times H\subset\mathbb Z_N}L^{GH}_{\varphi_G(i)\varphi_H(j)}$$
where the $L$ variables form dephased matrices $L^{GH}\in M_{G\times H}(\mathbb R)$. 
\end{theorem}

\begin{proof}
According to the above discussion, we just have to glue the various isotypic formulae coming from Proposition 3.6, by using Lemma 3.7. The gluing formula reads:
\begin{eqnarray*}
A_{i_1\ldots i_k,j_1\ldots j_k}
&=&A_{i_1j_1}\ldots A_{i_kj_k}\\
&=&\left(\sum_{r_1+s_1\leq a_1}L^{r_1s_1p_1}_{p_1^{a_1-r_1}i_1,p_1^{a_1-s_1}j_1}\ldots\sum_{r_k+s_k\leq a_k}L^{r_ks_kp_k}_{p_k^{a_k-r_k}i_k,p_k^{a_k-s_k}j_k}\right)\\
&=&\sum_{r_1+s_1\leq a_1}\ldots \sum_{r_k+s_k\leq a_k}L^{r_1s_1p_1}_{p_1^{a_1-r_1}i_1,p_1^{a_1-s_1}j_1}\ldots L^{r_ks_kp_k}_{p_k^{a_k-r_k}i_k,p_k^{a_k-s_k}j_k}
\end{eqnarray*}

Now, let us introduce the following variables:
$$L^{r_1\ldots r_k,s_1\ldots s_k}_{i_1\ldots i_k,j_1\ldots j_k}=L^{r_1s_1}_{i_1j_1}\ldots L^{r_ks_k}_{i_kj_k}$$

In terms of these new variables, the gluing formula reads:
$$A_{i_1\ldots i_k,j_1\ldots j_k}=\sum_{r_1+s_1\leq a_1}\ldots \sum_{r_k+s_k\leq a_k}L^{r_1\ldots r_k,s_1\ldots s_k}_{p_1^{a_1-r_1}i_1,\ldots p_k^{a_k-r_k}i_k,p_1^{a_1-r_1}j_1\ldots p_k^{a_k-r_k}j_k}$$

Together with the fact that the new $L$ variables form dephased matrices, in the sense of Definition 3.8 above, this gives the result.
\end{proof}

As an example, at $N=6$ the choices for the group $G\times H$ appearing in Theorem 3.10 are $\mathbb Z_1\times\mathbb Z_1,\mathbb Z_1\times\mathbb Z_2,\mathbb Z_1\times\mathbb Z_3,\mathbb Z_1\times\mathbb Z_6,\mathbb Z_2\times\mathbb Z_1,\mathbb Z_2\times\mathbb Z_3,\mathbb Z_3\times\mathbb Z_1,\mathbb Z_3\times\mathbb Z_2,\mathbb Z_6\times\mathbb Z_1$. According now to the dephasing conventions in Definition 3.8, which include our usual convention $\mathbb Z_{p^{-1}}=\emptyset$, these choices will produce $1+1+2+2+1+2+2+2+2=15$ variables, and we recover in this way the defect formula $d(F_6)=15$ coming from Proposition 2.3 (4).

As a first consequence, we have the following result, conjectured in \cite{ban}:

\begin{corollary}
The rationality conjecture holds for the Fourier matrices.
\end{corollary}

\begin{proof}
Indeed, the formula in Theorem 3.10 shows that for $H=F_N$ the rational defect, as constructed in Definition 2.8, counts the same variables as the usual defect.
\end{proof}

A few comments now regarding the generalized Fourier matrix case, $F=F_G$. In the isotypic case $G=\mathbb Z_{p^{a_1}}\times\ldots\times\mathbb Z_{p^{a_k}}$, according to Proposition 2.3 (2), we have:
\begin{eqnarray*}
d(F_G)
&=&\sum_{g\in G}\frac{N}{ord(g)}=\sum_{(g_1,\ldots,g_k)\in G}\frac{N}{\max(ord(g_1),\ldots,ord(g_k))}\\
&=&\sum_{r_1=0}^{a_1}\ldots\sum_{r_k=0}^{a_k}|\mathbb Z_{p^{r_1}}-\mathbb Z_{p^{r_1-1}}|\ldots |\mathbb Z_{p^{r_k}}-\mathbb Z_{p^{r_k-1}}|\cdot\frac{p^{a_1+\ldots+a_k}}{p^{\max(r_1,\ldots,r_k)}}
\end{eqnarray*}

The combinatorics here is obviously much more complicated. In fact, if we assume $a_1\leq a_2\leq\ldots\leq a_k$ then the defect formula from \cite{ban}, that we believe optimal, is:
$$d(F_G)=N\left(1+\sum_{r=1}^rp^{(k-r)a_{r-1}+(a_1+\ldots+a_{r-1})-1}(p^{k-r+1}-1)[a_r-a_{r-1}]_{p^{k-r}}\right)$$

Here $a_0=0$, and we use the standard notation $[a]_q=1+q+q^2+\ldots+q^{a-1}$.

As a conclusion, the following problem is open:

\begin{problem}
How does $\widetilde{T}_FC_N$ decompose, for a generalized Fourier matrix?
\end{problem}

Regarding now the affine tangent cone $T_F^\circ C_N$, we believe that, at least in the $G=\mathbb Z_N$ case, an iterated application of Theorem 2.7 should give the answer, but we don't have a proof for this fact. As for the tangent cone $T_FC_N$, very little is known here, see \cite{bbe}. We believe that the variables introduced in Theorem 3.10 can be of help in investigating the tangent cone, but so far we have no concrete results in this direction.

\section{Probabilistic aspects}

We have seen in the previous sections that for certain complex Hadamard matrices the defect is the number of 1 entries, and that the geometry basically comes from this. 

In general, the situation is much more complicated than that. One problem with the formula $d(H)=|1\in H|$ comes from the fact that the defect is insensitive to the equivalence relation in Definition 1.3, while the number of 1 entries is highly sensitive to it. So, unless $H$ is given to us in some natural, standard form, as is the case for instance with the Fourier matrices, we have to take into account all the quantities of type $|1\in\widetilde{H}|$, with $\widetilde{H}$ ranging over matrices which are equivalent to $H$.

\begin{conjecture}
We have the estimate $d(H)\leq\max|1\in\widetilde{H}|$, with the max ranging over all matrices $\widetilde{H}$ which are equivalent to $H$.
\end{conjecture}

Let us try now to formulate a finer version of this conjecture, roughly stating that ``$d(H)$ can be recaptured from the statistics of $|1\in\widetilde{H}|$, over the matrices $\widetilde{H}\simeq H$''.

There is an obvious problem with this latter statement, coming from the fact that the number $|1\in\widetilde{H}|$ is generically equal to $0$. In order to overcome this issue, one idea is to restrict attention to the Butson matrices, and to allow in Definition 1.3 only the multiplication on rows and columns by the corresponding roots of unity. 

More precisely, let us denote $\mathbb Z_s$ the group of $s$-th roots of unity, and by $C_N(s)$ the set of $N\times N$ complex Hadamard matrices having entries in $\mathbb Z_s$. We have then:

\begin{conjecture}
For $H\in C_N(s)$, with $s\in\mathbb N$ chosen to be minimal, we have
$$\min_{\widetilde{H}\simeq H}|1\in\widetilde{H}|\leq d(H)\leq\max_{\widetilde{H}\simeq H}|1\in\widetilde{H}|$$
with the min/max ranging over matrices of type $\widetilde{H}_{ij}=a_ib_jH_{ij}$, with $a_i,b_j\in\mathbb Z_s$.
\end{conjecture}

Observe that in this statement we have dropped the action of the symmetric group on the rows and columns of $H$, because this action leaves invariant both the defect and the number of 1 entries. As for the assumption that $s\in\mathbb N$ has to be minimal, this is of course in order for the lower bound to be non-trivial, because at $s>>0$ this minimum is $0$.

Now, let us go back to the comment following Conjecture 4.1. The quite vague statement formulated there can be now given a precise meaning, by using:

\begin{definition}
Let $H\in C_N(s)$ be a Butson matrix.
\begin{enumerate}
\item We define $\varphi:\mathbb Z_s^N\times\mathbb Z_s^N\to\mathbb N$ by $\varphi(a,b)=\#\{(i,j)|a_ib_jH_{ij}=1\}$.

\item We let $\mu$ be the probability measure on $\mathbb N$ given by $\mu(\{k\})=P(\varphi=k)$.
\end{enumerate}
\end{definition}

In this definition $P$ denotes the probability with respect to the uniform measure on the group $\mathbb Z_s^N\times\mathbb Z_s^N$. In other words, we regard $\varphi$ as a random variable over this group, and we denote by $\mu$ the distribution of this random variable:
$$\mu(\{k\})=\frac{1}{s^{2N}}\#\left\{(a,b)\in \mathbb Z_s^N\times\mathbb Z_s^N\Big|\varphi(a,b)=k\right\}$$

As a first observation, Conjecture 4.2 above can be reformulated as follows:

\begin{conjecture}
For $H\in C_N(s)$, with $s\in\mathbb N$ chosen to be minimal, we have 
$$d(H)\in\overline{supp(\mu)}^{\,conv}$$
where the measure $\mu$ is the one constructed in Definition 4.3 above.
\end{conjecture}

Summarizing, we have reached to a quite conceptual reformulation and generalization of our very first statement, Conjecture 4.1 above, at least in the Butson matrix case.

We will be back in a moment to this support problematics. But, let us formulate now yet another statement, which is our main conjecture on the subject:

\begin{conjecture}[Main conjecture]
For $H\in C_N(s)$, with $s\in\mathbb N$ chosen to be minimal, $d(H)$ can be recaptured from the knowledge of the associated measure $\mu$.
\end{conjecture}

As a first observation, this doesn't exactly generalize Conjecture 4.4. However, it is hard to imagine that a proof of this conjecture won't solve as well Conjecture 4.4.

In order to further comment on this conjecture, let us first do some computations. The very first problem concerns of course the support of $\mu$, and we have here:

\begin{proposition}
The support of $\mu$, with $s\in\mathbb N$ chosen to be minimal, is as follows: 
\begin{enumerate}
\item For $F_2$ we get $\{1,3\}$.

\item For $F_3$ we get $\{0,1,2,3,4,5\}$.

\item For $F_4$ we get $\{0,1,2,3,4,5,6,7,8\}$.

\item For $F_{2,2}$ we get $\{4,6,8,10,12\}$.

\item For $F_5$ we get $\{0,1,2,3,4,5,6,7,8,9\}$.
\end{enumerate}
\end{proposition}

\begin{proof}
Here is the proof for $F_4$, in logarithmic form, with the operations being those in Definition 1.3, and with the subscripts denoting the total number of 0 entries:
$$\begin{pmatrix}0&0&0&0\\ 0&1&2&3\\ 0&2&0&2\\ 0&3&2&1\end{pmatrix}_8\to
\begin{pmatrix}1&1&1&1\\ 0&1&2&3\\ 0&2&0&2\\ 0&3&2&1\end{pmatrix}_4\to 
\begin{pmatrix}2&1&1&1\\ 1&1&2&3\\ 1&2&0&2\\ 1&3&2&1\end{pmatrix}_1$$
$$\to\begin{pmatrix}2&1&2&1\\ 1&1&3&3\\ 1&2&1&2\\ 1&3&3&1\end{pmatrix}_0\to
\begin{pmatrix}2&1&2&0\\ 1&1&3&2\\ 1&2&1&1\\ 1&3&3&0\end{pmatrix}_2\to
\begin{pmatrix}0&1&2&0\\ 3&1&3&2\\ 3&2&1&1\\ 3&3&3&0\end{pmatrix}_3$$
$$\to\begin{pmatrix}0&0&2&0\\ 3&0&3&2\\ 3&1&1&1\\ 3&2&3&0\end{pmatrix}_5\to
\begin{pmatrix}0&0&0&0\\ 3&0&1&2\\ 3&1&3&1\\ 3&2&1&0\end{pmatrix}_6\to
\begin{pmatrix}0&0&3&0\\ 3&0&0&2\\ 3&1&2&1\\ 3&2&0&0\end{pmatrix}_7$$

The proof for the other matrices in the statement is similar.
\end{proof}

Perhaps the simplest general question regarding the support, and that we would like to raise here, is at $s=2$, and for the simplest Hadamard matrices, as follows:

\begin{problem}
What is $supp(\mu)$ for a Walsh matrix, $W_N$ with $N=2^k$?
\end{problem}

Let us discuss now the computation of the measure $\mu$ itself. For the first Walsh matrix $F_2\in C_2(2)$ it is easy to see that $\mu=\frac{1}{2}(\delta_1+\delta_3)$. More generally, we have:

\begin{proposition}
For $F_2\in C_2(s)$ with $s$ even we have
$$\mu=4\rho^{*3}-6\rho^{*2}+4\rho-\delta_0$$
where $\rho=\frac{s-1}{s}\delta_0+\frac{1}{s}\delta_1$ is the rescaled spectral measure of the main character of $\mathbb Z_s$.
\end{proposition}

\begin{proof}
We use the logarithmic writing. Consider the following matrix:
$$\widetilde{F}_2=\begin{pmatrix}i+a&i+b\\j+a&j+b+1\end{pmatrix}$$

Here the numbers $i,j,a,b$ range in the set $\{0,1,\ldots,s-1\}$, and are taken modulo $s$. What we have to do is to examine the number of 0 entries of $\widetilde{F}_2$, and compute the corresponding probability distribution $\mu$, which is supported on $\{0,1,2,3\}$.

A straightforward computation here gives the following formula: 
$$\mu=\frac{1}{s^3}((s^3-4s^2+6s-4)\delta_0+(4s^2-12s+12)\delta_1+(6s-12)\delta_2+4\delta_3)$$

But this is the formula in the statement, and we are done.
\end{proof}

Our second result concerns the second Walsh matrix, $W_4=F_{2,2}$. Here the computation at $s\in 2\mathbb N$ arbitrary looks quite complicated, but at $s=2$ we have: 

\begin{proposition}
For the second Walsh matrix, $F_{2,2}\in C_4(2)$, we have:
$$\mu=\frac{1}{32}(\delta_4+12\delta_6+6\delta_8+12\delta_{10}+\delta_{12})$$
\end{proposition}

\begin{proof}
We use the equivalence $F_{2,2}\simeq K_4$, where $K_4$ is the matrix having $-1$ on the diagonal and $1$ elsewhere. Now in logarithmic notation, we have:
$$(\widetilde{K}_4)_{ij}=a_i+b_j+\delta_{ij}$$

Thus if we want to compute the number of 1 entries, we have:
\begin{eqnarray*}
|1\in\widetilde{K}_4|
&=&\#\{(i,j)|i\neq j,a_i=b_j\}+\#\{(i,j)|i=j,a_i\neq b_j\}\\
&=&\#\{(i,j)|i\neq j,a_i=b_j\}+\#\{i|a_i\neq b_i\}\\
&=&\#\{(i,j)|a_i=b_j\}-\#\{i|a_i=b_i\}+\#\{i|a_i\neq b_i\}\\
&=&\#\{(i,j)|a_i=b_j\}+4-2\#\{i|a_i=b_i\}
\end{eqnarray*}

Now by writing down the $16\times 16$ tables for the two quantities appearing on the right, we obtain the explicit $16\times 16$ table of the values of $\varphi$, which gives the result.
\end{proof}

Observe that for $H\in C_N(2)$, computing the upper edge of the support of $\mu$ is the same as solving the corresponding Gale-Berlekamp game \cite{fsl}, \cite{rvi}. So, we have:

\begin{problem}[Gale-Berlekamp game] 
Given a Butson matrix $H\in C_N(s)$, consider the matrices $\widetilde{H}$ obtained from it by multiplying the rows and columns by roots of unity of order $s$. What is the maximal number of $1$ entries, over all these matrices $\widetilde{H}$?
\end{problem}

As a conclusion, the present results suggest that an interesting question would be that of connecting the various invariants of the complex Hadamard matrices to the Gale-Berlekamp game. We intend to explore this point of view in some future work.

\end{document}